\documentclass[11pt]{amsart}

\usepackage{amsaddr}
\usepackage{amscd,amssymb}
\usepackage{enumitem}

\title[Double calculus]{Double calculus}

\newtheorem{thm}{Theorem}
\newtheorem{prop}[thm]{Proposition}

\theoremstyle{definition}

\newtheorem{exa}[thm]{Example}
\newtheorem{rem}[thm]{Remark}
\newcommand{\D}{\Delta}

\newcommand{\N}{\mathbb{N}}
\newcommand{\R}{\mathbb{R}}

\setlist[itemize]{leftmargin=*}
\setlist[enumerate]{leftmargin=*}

\begin{document}

\author{Patrik Lundström}
\address{University West,
Department of Engineering Science, 
SE-461 86 Trollh\"{a}ttan, Sweden}

\email{patrik.lundstrom@hv.se}



\begin{abstract}
We present a streamlined, slightly modified version, 
in the two-variable situation,
of a beautiful, but not so well known, theory by B\"{o}gel
\cite{bogel1934,bogel1935}, already from the 1930s,
on an alternative higher 
dimensional calculus of real functions, a 
\emph{double calculus}, which includes 
two-variable extensions of many 
classical results from single variable calculus,
such as Rolle's theorem, 
 Lagrange's mean value theorem,
Cauchy's mean value theorem, 
Fermat's extremum theorem, the first derivative test, 
and the first and second fundamental theorems of calculus.
\end{abstract}

\maketitle

\section{Introduction}

The motivation for this article comes from the 
trivial observation that the difference operator $\Delta$
is connected to both the the derivative and the integral.
To be more precise, if $f$ is a single variable real function,
then $f$ has derivative 
$f'(a) = \lim_{b \to a} \D_a^b(f)/(b-a)$ at $a$,
if the limit exists exists, where 
$\D_a^b(f)$ denotes $f(b)-f(a)$, and
if $f$ is continuous, then, by the 
fundamental theorem of calculus, 
$\int_a^b f(x) \ dx = \D_a^b(F)$, where $F$ is
a primitive function of $f$.
A naive interpretation of this connection 
is that it should, in theory, \linebreak be possible
to obtain higher-dimensional analogues of the 
fundamental \linebreak theorem of calculus
by first defining a suitable difference operator 
given by the multiple integral, over suitable domains, 
and then, by reverse \linebreak engineering, use this difference 
operator to define a derivative so that the 
fundamental theorem of calculus holds.

It may come as a surprise to some readers, and
it certainly did so for the author of the present article, 
that B\"{o}gel \cite{bogel1934,bogel1935} 
already in the
1930s showed that it is indeed possible, 
in any finite number of variables, to successfully \linebreak
carry out such a program.
Since this theory should be of high 
interest to \linebreak students and 
instructors of calculus in several variables,
we have in this \linebreak article produced an
accessible and streamlined, slightly modified version,
of this program in the case of two variables,
that is, a theory of \emph{double calculus}.
Note that our notation, definitions, 
results and proofs, 
at times, somewhat differs from the approach by B\"{o}gel  
\cite{bogel1934,bogel1935}. In the presentation
we therefore carefully point out whenever that happens.

The domains that Bögel considers are the natural 
two-dimensional \linebreak analogues of intervals, namely 
\emph{double intervals} $[a,b] = [a_1,b_1] \times[a_2,b_2]$
where $a = (a_1,a_2)$ and $b = (b_1,b_2)$ are points
in $\R^2$ with $a_1 < b_1$ and $a_2 < b_2$.
If $f : [a,b] \to \R$ is a continuous two-variable function,
then, by iterated \linebreak integration, it follows that
$\iint_{[a,b]} f(x) \ dx_1 dx_2 = 
F(b_1,b_2) - F(b_1,a_2) - F(a_1,b_2) + F(a_1,a_2)$
where $F$ is a two-variable function with the 
property that the iterated partial derivatives 
$F_{12}$ and $F_{21}$ exist and are equal to $f$ on $[a,b]$.
Therefore he defines the \emph{double difference operator} by 
$
\D_a^b(f) = f(b_1,b_2) - f(b_1,a_2) - 
f(a_1,b_2) + f(a_1,a_2)
$
and he defines the
\emph{double} \linebreak \emph{derivative} of $f$ at $a$ by 
$f'(a) = \lim_{x_1 \to a_1; x_2 \to a_2} 
\D_a^b(f)/( (b_1-a_1) (b_2 - a_2) )$
when it exists. 

It turns out that the class of double differentiable 
functions thus obtained is much richer 
than it's one-dimensional counterpart.
In fact, the class
of double differentiable functions contains many 
examples of functions that are not partially 
differentiable and, in some cases, not even 
continuous. Indeed, if we pick \emph{any} functions
$g,h : \R \to \R$ and define $f : \R^2 \to \R$
by $f(x) = g(x_1) + h(x_2)$ for $x \in \R^2$, then 
$\D_a^b(f) = 0$ for all $a,b \in \R^2$ and hence
$f$ is \emph{double constant}, that is $f$ is 
double differentiable with $f' = 0$.
This has the unpleasant consequence that double 
differentiable functions 
may be \emph{unbounded} on compact subsets of $\R^2$.
Therefore, we can not expect to find a two-dimensional 
version of the Weierstrass extreme value theorem to hold
within this framework. 

Here is a detailed outline of this article.

In Section \ref{sec:doublefunctions}, we
fix the notation concerning double intervals
and double functions. We also show some 
elementary results that we need in subsequent 
sections.

In Section \ref{sec:doublecontinuity},
we define the class of double
continuous functions. We show that 
a function which is double continuous on an interval
is automatically \linebreak \emph{globally} double continuous
on that double interval. 
We also show a \linebreak
continuity result concerning the double 
difference map that we need in the following sections.

In Section \ref{sec:doubledifferentiablility}, 
we define (signed) double limits and
(signed) double \linebreak derivatives. 
In analogue with the single variable situation, 
we show that double differentiable functions
are double continuous.
Then we show double versions of Rolle's theorem,
the mean value theorem, Fermat's theorem 
and the first derivative test.
At the end of this section, 
we introduce double primitive functions.

In Section \ref{sec:doubleintegrability},
we define the double Newton integral.
Using the double mean value theorem, we 
obtain a mean value theorem for double 
Newton integrals.
After that, we connect the double Newton integral
to the Riemann double integral in 
the first and second double fundamental 
theorems of calculus.
At the end of this section, we introduce
improper double Newton integrals.
We also discuss some examples of double 
integrals over
non-rectangular regions.

In Section \ref{sec:triplecalculus}, we
discuss B\"{o}gels extensions of the results in this 
article to higher dimensions, that is
triple calculus, quadruple calculus and beyond.
We also discuss the possibility for a  
higher-dimensional version
of Schwarz's theorem and 
Darboux's theorem.

\section{Double functions}\label{sec:doublefunctions}

In this section, we
fix the notation concerning double intervals
and double functions. We also show some 
elementary results that we need in subsequent 
sections (see Proposition \ref{propsplit} and
Proposition \ref{subdivision}).

Let $\N$ denote the set of positive integers.
We let $\R$ denote the set of real numbers and
we put $\R^2 := \R \times \R$.
Let $\R_+$ and $\R_-$ denote the set
of positive real numbers and the set of negative
real numbers respectively;
we put $\R_{++}^2 := \R_+ \times \R_+$,
$\R_{+-}^2 := \R_+ \times \R_-$,
$\R_{-+}^2 := \R_- \times \R_+$ and
$\R_{--}^2 := \R_- \times \R_-$.

Suppose that $a \in \R^2$. We write $a = (a_1,a_2)$ 
for $a_1,a_2 \in \R$.
More generally, if $A \subseteq \R^2$, then we put
$A_1 = \{ a_1 \mid a \in A \}$ and
$A_2 = \{ a_2 \mid a \in A \}$.
Suppose that $b \in \R^2$. Then we write 
$a \sim b$ if $a_1 = b_1$ or $a_2 = b_2$; 
$a \nsim b$ if $a_1 \neq b_1$ and $a_2 \neq b_2$;
$a < b$ if $a_1 < b_1$ and $a_2 < b_2$;
$a \leq b$ if $a_1 \leq b_1$ and $a_2 \leq b_2$.

By a \emph{double interval} we mean a subset $I$ of $\R^2$
of the form $I_1 \times I_2$ where $I_1$ and $I_2$
are intervals in $\R$.
If $I_1$ and $I_2$ are open (closed, compact),
then we say that $I$ is \emph{double open 
(double closed, double compact)}.
If $a,b \in \R^2$ and $a \leq b$, then we put
$(a,b) = \{ x \in \R^2 \mid a < x < b \}$
and
$[a,b] = \{ x \in \R^2 \mid a \leq x \leq b \}$.

By a \emph{double $\delta$-neighbourhood of $a$} 
we mean a set of the form
$
D(a,\delta) := (a-\delta,a+\delta)
$
for some $\delta \in \R_{++}^2$.
The \emph{signed} double $\delta$-neighbourhoods
$D_{++}(a,\delta)$, $D_{-+}(a,\delta)$, 
$D_{--}(a,\delta)$ and $D_{+-}(a,\delta)$ are
defined as the intersection of $D(a,\delta)$ with, 
respectively, the closure of the first, the second, 
the third and the fourth quadrant in $\R^2$.
By a \emph{double punctured $\delta$-neighbourhood of $a$}
we mean a set of the form
$
P(a,\delta) := \{ x \in D(a,\delta) \mid x \nsim a \}
$
for some $\delta \in \R_{++}^2$.
The \emph{signed} double punctured $\delta$-neighbourhoods
$P_{++}(a,\delta)$, $P_{-+}(a,\delta)$, 
$P_{--}(a,\delta)$ and $P_{+-}(a,\delta)$ are
defined as the intersection of $D(a,\delta)$ with, 
respectively, the first, the second, 
the third and the fourth quadrant in $\R^2$.

Let $f$ denote a real-valued function
with a domain $D(f)$ which is a subset of $\R^2$.
In that case, we say that $f$ is a 
\emph{double function}.
If $a,b \in \R^2$ are chosen so that
$(b_1,b_2),(b_1,a_2),(a_1,b_2),(a_1,a_2) \in D(f)$, 
then we define the 
\emph{double difference of $f$ from $a$ to $b$} 
as the real number:
\[
\D_a^b(f) := f(b_1,b_2) - f(b_1,a_2) - f(a_1,b_2) + f(a_1,a_2).
\]
If $I$ is a double interval contained in $D(f)$, then
we say that $f$ is  \emph{double constant} on $I$
if $\D_a^b(f) = 0$ for all $a,b \in I$.

\begin{prop}\label{propsplit}
Suppose that $f$ is a double function which is defined
on a double interval $I$.
Then $f$ is double constant on $I$
if and only if
there are functions $g : I_1 \to \R$ and 
$h : I_2 \to \R$ with
$f(x) = g(x_1) + h(x_2)$ for $x \in I$.
\end{prop}

\begin{proof}
Suppose $f$ is double constant on $I$.
Take $a \in I$.
Define $g : I_1 \to \R$ by $g(s) = f(s,a_2)$ 
for $s \in I_1$
and define $h : I_2 \to \R$ by 
$h(t) = f(a_1,t) - f(a_1,a_2)$ for $t \in I_2$.
If $x \in I$, then:
\[
\begin{array}{rcl} 
f(x_1,x_2) & = &  f(x) - 0 
     = f(x_1,x_2) - \D_a^x(f) \\[5pt]
     &=& f(x_1,x_2) - f(x_1,x_2) - f(a_1,a_2) + f(x_1,a_2) + f(a_1,x_2) \\[5pt]  
     &=& g(x_1) + h(x_2).
\end{array}
\]
Now suppose that there are functions $g : I_1 \to \R$ and 
$h : I_2 \to \R$ with $f(x) = g(x_1) + h(x_2)$ for $x \in I$. 
If $a,b \in I$, then:
\[
\begin{array}{rcl}
\D_a^b(f) & = & f(b_1,b_2) + f(a_1,a_2) - 
f(b_1,a_2) - f(a_1,b_2) \\[5pt]
&=&  g(b_1) + h(b_2) + g(a_1) + h(a_2) \\[5pt]
&-& \left( g(b_1) + h(a_2) + g(a_1) + h(b_2) \right)
\ = \ 0.  
\end{array}
\]
Thus, $f$ is double constant on $I$.
\end{proof}

For future reference, we record some 
properties of the double difference.

\begin{prop}\label{subdivision}
Suppose that $f$ is a double function which is defined on
a double interval containing the points $a$, $b$ and $x$.
Then:
\begin{enumerate}[label={\rm (\alph*)}]

\item  $\D_a^a(f) = 0$; \ $\D_a^b(f) = \D_b^a(f)$; \
$\D_a^b(f) = -\D_{(a_1,b_2)}^{(b_1,a_2)}(f)$;

\item $\D_a^b(f) = \D_a^{(x_1,b_2)}(f) + 
\D_{(x_1,a_2)}^b(f)$;

\item $\D_a^b(f) = \D_a^{(b_1,x_2)}(f) + 
\D_{(a_1,x_2)}^b(f)$;

\item $\D_a^b(f) = \D_a^x(f) + \D_x^b(f) + 
\D_{(a_1,x_2)}^{(x_1,b_2)}(f) +  
\D_{(x_1,a_2)}^{(b_1,x_2)}(f)$.

\end{enumerate}
\end{prop}

\begin{proof}
These properties follow immediately from 
the definition of $\D$.
\end{proof}

\section{Double continuity}\label{sec:doublecontinuity}

In this section, we define the class of double
continuous functions and the class of globally
double continuous functions. 
We show that 
a function which is double continuous on an interval
is automatically \emph{globally} double continuous
on that double interval 
(see Proposition \ref{propgloballydouble}). 
We also show a 
continuity result (see Proposition \ref{propcontinuity}), 
concerning the double 
difference map that we need in the next section.

Let $a \in \R^2$ and suppose that $f$ is a double function. 
We say that $f$ is \emph{double continuous at $a$} 
if there is a double open interval $I$
with $a \in I \subseteq D(f)$ such that
$\forall x_1 \in I_1$ 
$\lim_{x_2 \to a_2} \D_a^x(f) = 0$ and
$\forall x_2 \in I_2$ 
$\lim_{x_1 \to a_1} \D_a^x(f) = 0$.

Now we define \emph{signed} double 
continuity at a point. 
We say that $f$ is \emph{$+-$double continuous at $a$} 
if there is $\delta \in \R_{++}^2$
with $a \in I:=D_{+-}(a,\delta) \subseteq D(f)$ such that
$\forall x_1 \in I_1$ 
$\lim_{x_2 \to a_2^-} \D_a^x(f) = 0$ and
$\forall x_2 \in I_2$ 
$\lim_{x_1 \to a_1^+} \D_a^x(f) = 0$.
Analogously, $++$double continuity, 
$-+$double continuity and 
$--$double continuity are defined.

Let $I$ be a double interval 
such that $I \subseteq D(f)$.
If $f$ is double continuous at every $a \in I$,
then we say that $f$ is \emph{double continuous on $I$}.
In this definition it is understood that if $a$
is a boundary point of $I$, then 
by double continuity at $a$ 
we mean this in the sense of the 
signed double continuity defined above.
Namely, suppose, for instance, that 
$I = (a_1 , b_1] \times (a_2 , b_2]$.
By double continuity at the point $b$, we mean 
$--$double continuity, and by
double continuity at a point $(b_1,t)$, where $a_2 < t < b_2$,
we mean both $-+$double continuity and 
$--$double continuity.

\begin{prop}\label{contaimpliesdconta}
Suppose that $f$ is a double function. 
Let $I$ be a double interval.

\begin{enumerate}[label={\rm (\alph*)}]

\item If $f$ is continuous on $I$,
then $f$ is double continuous on $I$.

\item If $f$ is double constant on $I$, then $f$ is
double continuous on $I$.

\end{enumerate}
\end{prop}

\begin{proof}
This is clear.
\end{proof}

\begin{exa}\label{ex:doubleconstant}
The class of double continuous functions on $\R^2$
contains many examples of functions which are not
continuous. Indeed, suppose that $g$ and $h$ are
functions $\R \to \R$ such that $g$ is continuous but
$h$ is not. Put $f(x) = g(x_1) + h(x_2)$ for $x \in \R^2$. 
Then, clearly, $f$ is not continuous.
However, by Proposition \ref{propsplit} and 
Proposition \ref{contaimpliesdconta}(b), $f$
is double continuous.
\end{exa}

\begin{exa}\label{ex:x1x2}
Define the double function $f$ by
$f(x) = x_1^{x_2} \cdot x_2^{x_1}$,
for $x > 0$, and $f(x) = 0$, if $x \leq 0$.
Note that $f$ is not double 
constant since $\D_0^x(f) = f(x) \neq 0$ for all 
$x > 0$. Since:
\[
f(t,t) = t^t \cdot t^t = 
t^{2t} = e^{2t \ln(t)} \to e^0 = 1 \neq 0 = f(0,0),
\]
as $t \to 0^+$, it follows that $f$ is not 
continuous at $0$. 
However, $f$ is double continuous at $0$.
In fact, if $x_1 > 0$, then: 
\[
\D_0^x(f) =  f(x) = x_1^{x_2} \cdot x_2^{x_1} 
\to 1 \cdot 0 = 0 = f(0,0), 
\]
as $x_2 \to 0^+$; if $x_2 > 0$, then:
\[
\D_0^x(f) = f(x) = x_1^{x_2} \cdot x_2^{x_1}
\to 0 \cdot 1 = 0 = f(0,0), 
\]
as $x_1 \to 0^+$.
\end{exa}

\begin{rem}
Our definition of double continuity
is different from Bögel's \cite{bogel1934} 
definition of this concept. Indeed, 
he defines $f$ to be continuous at $a$ if
$\lim_{x \to a} \D_a^x(f) = 0$.
Example \ref{ex:x1x2} shows that there are double continuous functions in our sense that are not double continuous in the sense of Bögel. On the other hand, the function 
$f(x) = x_1^2 + x_2^2$, for $x \nsim 0$,
and $f(x) = 0$, for $x \sim 0$, 
is continuous at 0 in the sense of Bögel but not in our sense.
We have chosen to define double continuity as weak 
as possible but so that it still is implied by double 
differentiability (see Proposition \ref{ddiffimpliesdcont}). 
\end{rem}

Suppose that $f$ is a double function. 
Let $I$ be a double interval such that $I \subseteq D(f)$.
We say that $f$ is {\it globally double continuous on $I$} if
$\forall c,d \in I_1$ $\forall e \in I_2$ 
$\lim_{x_2 \to e; x_2 \in I_2} 
\D_{(d,e)}^{(c,x_2)}(f) = 0$, and
$\forall c,d \in I_2$ $\forall e \in I_1$ 
$\lim_{x_1 \to e; x_1 \in I_1} 
\D_{(e,d)}^{(x_1,c)}(f)$ $= 0$.
To the knowledge of the author of the 
present article, Bögel \cite{bogel1934,bogel1935}
does not define any concept resembling 
global double continuity.

\begin{prop}\label{propgloballydouble}
Let $I$ be a double interval and suppose that 
$f$ is a double function such that $I \subseteq D(f)$.
Then $f$ is double continuous on $I$ if and only if
$f$ is globally double continuous on $I$.
\end{prop}

\begin{proof}
 The ``if'' statement is clear.
Now we show the ``only if'' statement.
Suppose that $f$ is double continuous on 
$I$. Take $a,b \in I$ with $a<b$.
Take $e \in [a_2,b_2]$ and $c,d \in [a_1,b_1]$ with $c < d$.
Put $J = [c,d]$. The case when $d < c$ 
is symmetrical and is therefore left to the reader.
For all $z \in J$ choose $\epsilon_z > 0$ such that 
for all $y \in (z - \epsilon_z , z + \epsilon_z)$ the map 
$[a_2,b_2] \ni x_2 \mapsto f(y,x_2)-f(c,x_2)$
is continuous at $x_2 = e$. Since the open intervals 
$( z - \epsilon_z/2 , z + \epsilon_z/2 )$, for $z \in J$, 
is an open cover of
the compact interval $J$ we can choose a finite subcover 
$\{ ( z_i - \epsilon_{z_i}/2 , z_i + 
\epsilon_{z_i}/2 ) \}_{i=0}^n$ of $J$. 
We may assume that 
$a_1 = z_0 < z_1 < \cdots < z_{n-1} < z_n = x_1$. 
We may also assume that for all $i,j \in \{ 0,\ldots,n \}$
if $( z_i - \epsilon_{z_i}/2 , z_i + \epsilon_{z_i}/2 ) 
\subseteq ( z_j - \epsilon_{z_j}/2 , z_j + \epsilon_{z_j}/2 )$,
 then $i=j$.
It follows that:
\[
\begin{array}{rcl}
f(x_1,x_2)-f(a_1,x_2) &=& 
\sum_{i=1}^n f(z_i,x_2) - f(z_{i-1},x_2) \\[5pt]
 &\to&  \sum_{i=1}^n f(z_i,a_2) - f(z_{i-1},a_2) \\[5pt]
 &=&  f(x_1,a_2) - f(a_1,a_2),
\end{array}
\]
as $x_2 \to a_2$,
since for all $i \in \{ 1,\ldots,n \}$, 
$z_{i-1} \in ( z_i - \epsilon_{z_i} , z_i + \epsilon_{z_i} )$
or $z_i \in ( z_{i-1} - \epsilon_{z_{i-1}} , z_{i-1} + \epsilon_{z_{i-1}} )$.
The calculation involving the second limit 
is completely analogous to the above calculation 
and is left to the reader.
\end{proof}

The next result is \cite[Satz 12]{bogel1934}.
To prove this we use our notion of global 
double continuity, whereas B\"{o}gel loc. cit. resorts to an
ad hoc argument (although in a more general context).

\begin{prop}\label{propcontinuity}
Let $a,b,h \in \R^2$ where $a < b$. 
Suppose that $f : [a,b] \to \R$ is a double 
continuous function.
Then the map $x \mapsto \D_x^{x+h}$ is continuous 
at all $x \in (a,b)$ for which $x+h \in (a,b)$.
\end{prop}

\begin{proof}
It is enough to show that the map $x \mapsto \D_x^{x+h}$ 
is continuous separately in the variables $x_1$ and $x_2$.
By Proposition \ref{propgloballydouble} it follows that:
\[
\begin{array}{rcl}
\D_x^{x+h}(f) - \D_{(a_1,x_2)}^{(a_1+h_1,x_2+h_2)} 
&=& f(x_1+h_1,x_2+h_2)-f(x_1+h_1,x_2) \\[5pt]
& & - \ ( f(a_1+h_1,x_2+h_2) - f(a_1+h_1,x_2) ) \\[5pt]
& & - \ ( f(x_1,x_2+h_2) - f(x_1,x_2) ) \\[5pt]
& & + \ f(a_1,x_2+h_2) - f(a_1,x_2) \ \to \ 0, 
\end{array}
\]
as $x_1 \to a_1$, and:
\[
\begin{array}{rcl}
\D_x^{x+h}(f) - \D_{(x_1,a_2)}^{(x_1+h_1,a_2+h_2)}
&=& f(x_1+h_1,x_2+h_2)-f(x_1+h_1,x_2) \\[5pt]
& & - \ ( f(x_1+h_1,a_2+h_2) - f(x_1+h_1,a_2) ) \\[5pt]
& &  - \ ( f(x_1,x_2+h_2) - f(x_1,x_2) ) \\[5pt]
& &  + \ f(x_1,a_2+h_2) - f(x_1,a_2) \ \to \ 0, 
\end{array}
\]
as $x_2 \to a_2$.
\end{proof}

\begin{prop}\label{propintermeiate}
Let $a,b \in \R^2$ where $a < b$.
Suppose that $f$ is a double function 
which is continuous on $[a,b]$.
Let $D$ denote the smallest closed interval in $\R$
containing $f(a_1,a_2)$, $f(b_1,b_2)$, $f(a_1,b_2)$ 
and $f(b_1,a_2)$.
If $d$ is an interior point in $D$, 
then there exists $c \in (a,b)$ with $f(c) = d$.
\end{prop}

\begin{proof}
From the assumptions it follows that there are $a',b' \in (a,b)$
with $f(a_1',a_2') < d < f(b_1',b_2')$.
Let $L$ denote the line segment from $a'$ to $b'$.
Since $L$ is connected and $f$ is continuous,
$f(L)$ is an interval. 
Thus, there is $c \in L$ with $f(c)=d$. 
Since $c \in L \subsetneq [a,b]$
it follows that $c \in (a , b)$.
\end{proof}

\section{Double differentiability}
\label{sec:doubledifferentiablility}

In this section, we define (signed) double limits and
(signed) double derivatives. 
In analogue with the single variable situation,
we show that double differentiable functions
are double continuous (see Proposition \ref{ddiffimpliesdcont}).
Thereafter, we show a double Rolle's theorem
(see Proposition \ref{doublerolle}),
a double Lagrange's mean value theorem
(see Proposition \ref{meanvaluethm}),
a double Cauchy's mean value theorem 
(see Proposition \ref{meanvaluethmcauchy}),
a double Fermat's theorem 
(see Proposition \ref{propfermat})
and a double first derivative test
(see Proposition \ref{propsigns}).
At the end of this section, we define
double primitive functions.

Suppose that $a \in \R^2$, $L \in \R$ and that 
$f$ is a double function.
We say that $f$ has \emph{double limit $L$ as $x$
approaches $a$} 
if for every $\epsilon > 0$ there is
$\delta \in \R_{++}^2$ with  
$P(a,\delta) \subseteq D(f)$
and $| f(x_1,x_2) - L | < \epsilon$ whenever
$x \in P(a,\delta)$. 
In that case, we write 
$\lim_{x \leadsto a} f(x_1,x_2) = L$ or
$f(x_1,x_2) \to L$ as $x \leadsto a$.

Now we define \emph{signed} double limits.
We say that $f$ has \emph{$+-$double limit 
$L$ as $x$ approaches $a$} 
if for every $\epsilon > 0$ there is
$\delta \in \R_{++}^2$ with 
$P_{+-}(a,\delta) \subseteq D(f)$,
and $| f(x_1,x_2) - L | < \epsilon$ whenever
$x \in P_{+-}(a,\delta)$. 
In that case, we write 
$\lim_{x \leadsto a^{+-}} f(x_1,x_2) = L$ or
$f(x_1,x_2) \to L$ as $x \leadsto a^{+-}$.
Analogously, $++$double limits, 
$-+$double limits and 
$--$double limits are defined.

More generally, one may analogously define (signed)
double limits when $a$ and $L$ belong to
$\overline{\R}^2 = \overline{\R} \times \overline{\R}$,
where $\overline{\R}$ denotes the 
affinely extended real number system 
$\R \cup \{ \infty,-\infty \}$.
We leave the details of these definitions to the 
reader.

Suppose that $a,b \in \R^2$ and that $a \nsim b$. 
If $f$ is defined at $(a_1,a_2)$, $(a_1,b_2)$,
$(b_1,a_2)$ and $(b_1,b_2)$, then we define the 
\emph{double mean slope of $f$ from $a$ to $b$}
to be the quotient:
\[
m_a^b(f) := \frac{ \D_a^b(f) }{ (b_1 - a_1)(b_2 - a_2) }.
\]
We say that $f$ is \emph{double differentiable} at $a$ if 
there is a double open interval $I$ with 
$a \in I \subseteq D(f)$ and
the double limit $\lim_{x \leadsto a} m_a^x(f)$ exists. 
In that case, we let $f'(a)$ denote this limit and we say that
$f'(a)$ is the \emph{double derivative} of $f$ at $a$
(cf. \cite{trahan1969}). 
Using signed double limits, we can analogously
define the $++$double derivative $f'_{++}(a)$,
the $+-$double derivative $f'_{+-}(a)$,
the $-+$double derivative $f'_{-+}(a)$ and the
$--$double derivative $f'_{--}(a)$.

Let $I$ be a double interval 
such that $I \subseteq D(f)$.
If $f$ is double differentiable at every $a \in I$,
then we say that $f$ is \emph{double differentiable on $I$}.
In this definition it is understood that when $a$
is a boundary point of $I$, then 
by double differentiable at $a$ 
we mean this in the sense of the 
signed double differentiability defined above.
Namely, suppose, for instance, that 
$I = [a_1 , b_1) \times (a_2 , b_2]$.
By differentiability at the point $(a_1,b_2)$, we mean 
$+-$double differentiability, and by
double differentiability at a point 
$(a_1,t)$, where $a_2 < t < b_2$,
we mean both $++$double differentiability and 
$+-$double differentiability.

The next result, more commonly known 
as Schwarz's theorem, Clairaut's theorem, or Young's theorem, 
is well-known (see e.g. \cite[Theorem 9.40]{rudin1976}).
We have, nevertheless, chosen to include a proof of it 
since it involves, in a natural way, 
use of the double derivative.
Note that for iterated partial derivatives, we will 
use the convention that $f_{12}$ and $f_{21}$ 
denote $(f_1)_2$ and $(f_2)_1$ respectively.

\begin{prop}\label{schwarztheorem}
Suppose that $f$ is a double function with the property that
the mixed partial derivatives $f_{12}$ and $f_{21}$ 
exist and are continuous at $a \in \R^2$.
Then $f$ is double differentiable at $a$ and 
$f'(a) = f_{12}(a) = f_{21}(a)$.
\end{prop}

\begin{proof}
The assumptions imply that $f$ and the partial
derivatives $f_{12}$, $f_{21}$,
$f_1$ and $f_2$ are defined in some double 
open interval $I$ containing $a$.
Take $h \in \R^2$ with $h \nsim 0$.
Define functions $u$ and $v$ by:
\[
u(x_1) = f(x_1 , a_2 + h_2) - f(x_1,a_2), 
\]
for $x_1 \in I_1$, and:
\[
v(x_2) = f(a_1 + h_1, x_2) - f(a_1 , x_2), 
\]
for $x_2 \in I_2$.
Repeated use of the single variable mean value theorem yield
$\theta_1,\theta_2 \in (0,1)$ such that:
\begin{eqnarray*}
\Delta_a^{a+h}(f) &=& u(a_1 + h_1) - u(a_1) = h_1 u'(a_1 + \theta_1 h_1) \\[5pt]
&=& h_1 ( f_1(a_1 + \theta_1 h_1 , a_2 + h_2) - f_1( a_1 + \theta_1 h_1, a_2) ) \\[5pt]
&=& h_1 h_2 f_{12}( a_1 + \theta_1 h_1 , a_2 + \theta_2 h_2).
\end{eqnarray*}
Thus,
$m_a^{a+h}(f) =  f_{12}( a_1 + \theta_1 h_1 , a_2 + \theta_2 h_2)
\to f_{12}(a_1,a_2)$ as $h \leadsto 0$.
Similarly, there exist $\theta_3,\theta_4 \in (0,1)$ with:
\begin{eqnarray*}
\Delta_a^{a+h}(f) &=& v(a_2 + h_2) - v(a_2) = h_3 v'(a_2 + \theta_3 h_2) \\[5pt]
&=& h_3 ( f_2(a_1+h_1 , a_2 + \theta_3 h_2) - f_2( a_1 , a_2 + \theta_3 h_2) ) \\[5pt]
&=& h_1 h_2 f_{21}( a_1 + \theta_4 h_1 , a_2 + \theta_3 h_2).
\end{eqnarray*}
Hence,
$m_a^{a+h}(f) =  
f_{21}( a_1 + \theta_4 h_1 , a_2 + \theta_3 h_2) 
\to f_{21}(a_1,a_2)$ as $h \leadsto 0$.
Thus, $f$ is double differentiable at $a$ and 
$f'(a) = f_{12}(a) = f_{21}(a)$.
\end{proof}

\begin{exa}
By Proposition \ref{schwarztheorem}, all sufficiently
smooth double functions are double differentiable.
However, the class of double differentiable \linebreak
functions contains examples of functions 
which are not even partially \linebreak differentiable.
In fact, it is clear that all double constant functions are double 
differentiable with double derivative equal to zero
everywhere.
Therefore, the class of double diffe\-rentiable functions
even contains many examples of everywhere 
discontinuous
functions (see Example \ref{ex:doubleconstant}).
\end{exa}

\begin{exa}
It is easy to construct examples of double functions 
that are not double 
constant but double differentiable at a point
where an iterated partial derivative fails to exist.
Namely, suppose that $g$ and $h$ are single variable 
functions defined on $\R$.
Furthermore, suppose that $g(x_1) \to 0$ as $x_1 \to 0$;
$g(x_1) \neq 0$ for non-zero $x_1 \in \R$;
$h(x_2)$ is bounded near the origin;
in any open interval around the origin, 
there exists $x_2$ such that $h(x_2) \neq 0$;
$\lim_{x_2 \to 0} h(x_2)$ does not exist.
For instance, we can choose 
$g(x_1) = x_1$ and
$h(x_2) = \sin(1/x_2)$, for $x_2 \neq 0$, and 
$h(0)=0$.
Now put $f(x) = x_1 g(x_1) \cdot x_2 h(x_2)$ for $x \in \R^2$.
Then, in any double open interval containing the 
origin, there exists $x$ such that 
$\D_0^x(f) = f(x) \neq 0$.
Also, clearly, the double derivative $f'(0,0)$ exists 
and is equal to zero. However, for any fixed 
nonzero $x_1 \in \R$, 
$f_2(x_1,0) = \lim_{x_2 \to 0} f(x)/x_2 = 
\lim_{x_2 \to 0} x_1 g(x_1) h(x_2)$ does not exist.
Thus, the iterated partial derivative $f_{21}(0,0)$ 
does not exist. 
\end{exa}

In analogy with the single variable situation,
there is a ``first order'' \linebreak approximation to $\D$:

\begin{prop}\label{ddequivalent}
Suppose that $f$ is a double function and that $a \in \R^2$.
Then $f$ is double differentiable at $a$
if and only if there is $d \in \R$, a double 
neighbourhood $D(a,\delta)$, where $f$ is defined,
and a function 
$\rho : P(a,\delta) \to \R$ such that 
$\D_a^x(f) = (x_1 - a_1)(x_2 - a_2)d + 
(x_1-a_1) (x_2-a_2) \rho(x)$,
for $x \in P(a,\delta)$, and
$
\lim_{x \leadsto a} \rho(x) = 0
$.
In that case, $f'(a) = d$.
\end{prop}

\begin{proof}
This follows immediately from the definition of $\D$.
\end{proof}

\begin{prop}\label{ddiffimpliesdcont}
Suppose that $f$ is a double function which is
defined at $a \in \R^2$.
If $f$ is double differentiable at $a$,
then $f$ is double continuous at $a$. 
\end{prop}

\begin{proof}
Suppose that $f$ is double differentiable at $a$.
Take $\delta$ and $\rho$ satisfying the conditions
in the formulation of Proposition \ref{ddequivalent}.
If $x \nsim a$, then:
\[
\begin{array}{rcl}
f(x_1,x_2)-f(a_1,x_2) & = &  \D_a^x(f) + f(x_1,a_2)-f(a_1,a_2) \\[5pt]
& = & (x_1 - a_1)(x_2 - a_2)d + (x_1-a_1) (x_2-a_2) \rho(x) \\[5pt]
& + & f(x_1,a_2)-f(a_1,a_2) \ \to \ f(x_1,a_2)-f(a_1,a_2),
\end{array}
\] 
as $x_2 \to a_2$, and:
\[
\begin{array}{rcl}
f(x_1,x_2)-f(x_1,a_2) & = &  
\D_a^x(f) + f(a_1,x_2)-f(a_1,a_2) \\[5pt]
& = & (x_1 - a_1)(x_2 - a_2)d + (x_1-a_1) (x_2-a_2) \rho(x) \\[5pt]
& + & f(a_1,x_2)-f(a_1,a_2) \ \to \ f(a_1,x_2)-f(a_1,a_2),
\end{array}
\]
as $x_1 \to a_1$.
\end{proof}

We now proceed to show double versions of Rolle's theorem
and various versions of the mean value theorem. 
To this end, we need
two propositions. \linebreak
We adapt, to the situation at hand, an approach
which originally was \linebreak
invented by Cauchy (see \cite[p. 169]{grabiner1981})
for the single variable situation 
and then corrected and clarified by Plante \cite{plante2017}.

\begin{prop}\label{propsqueeze} 
Suppose that $a,b \in \R^2$, $a < b$, and that $x \in (a,b)$.
Let $f : [a,b] \to \R$ be a double function.  
Consider the following four real numbers: 
{\small
\[ 
m_1 := m_a^x(f), \ m_2 := m_x^b(f), \
m_3 := \frac{\D_{(a_1,x_2)}^{(x_1,b_2)}(f)}{(x_1-a_1)(b_2-x_2)}, 
\ 
m_4 := \frac{\D_{(x_1,a_2)}^{(b_1,x_2)}(f)}{(b_1-x_1)(x_2-a_2)}.
\] }

\noindent Then either (i) all of them are equal to $m_a^b(f)$, 
or (ii) at least one of them is greater than $m_a^b$ and 
at least one of them is less than $m_a^b$.
\end{prop}

\begin{proof}
Suppose that (i) does not hold. 
Seeking a contradiction,  
suppose that $m_a^b(f) \leq m_i$ for all $i \in \{ 1,2,3,4 \}$
with strict inequality for at least one index.
By Proposition \ref{subdivision}(d), we get that:
\[
\begin{array}{rcl}
\D_a^b(f) & = & m_a^b(f) (b_1-a_1)(b_2-a_2) \\[5pt]
& = & m_a^b(f) (x_1-a_1)(x_2-a_2) + 
m_a^b(f) (b_1-x_1)(b_2-x_2) \\[5pt]
& + & m_a^b(f) (x_1-a_1)(b_2-x_2) + 
m_a^b(f) (b_1-x_1)(x_2-a_2) \\[5pt]
& < & m_1 (x_1-a_1)(x_2-a_2) + m_2 (b_1-x_1)(b_2-x_2) \\[5pt]
& + & m_3 (x_1-a_1)(b_2-x_2) + m_4 (b_1-x_1)(x_2-a_2) \\[5pt]
& = & \D_a^x(f) + \D_x^b(f) + \D_{(a_1,x_2)}^{(x_1,b_2)}(f) +  \D_{(x_1,a_2)}^{(b_1,x_2)}(f) \ = \ \D_a^b(f)
\end{array}
\]
and hence the contradiction $\D_a^b(f) < \D_a^b(f)$.
Similarly, if we assume that $m_a^b(f) \geq m_i$, 
for all $i \in \{ 1,2,3,4 \}$,
with strict inequality for at least one index,
then we get the contradiction $\D_a^b(f) > \D_a^b(f)$.
Thus (ii) holds.
\end{proof}

\begin{prop}\label{propsqueezeagain}
Suppose that $a,b \in \R^2$ satisfy $a < b$.
Let $f : (a,b) \to \R$ be a double function which is double differentiable at $c \in (a,b)$. Let
$
a(1) \leq a(2) \leq a(3) \leq \cdots
$
and
$
b(1) \geq b(2) \geq b(3) \geq \cdots
$
be sequences in $(a,b)$ with
$
\lim_{n \to \infty} a(n) = \lim_{n \to \infty} b(n) = c
$
and satisfying one of the following properties:
\begin{itemize}

\item[]{\rm (i)} $\forall n \in \N$ $a(n) < c < b(n)$;

\item[]{\rm (ii)} $\forall n \in \N$ $a(n) < c$ and $\exists N \in \N$ $\forall n \geq N$ $c = b(n)$;

\item[]{\rm (iii)} $\forall n \in \N$ $c < b(n)$ and $\exists N \in \N$ $\forall n \geq N$ $a(n) = c$.

\end{itemize}
Then $\lim_{n \to \infty} m_{a(n)}^{b(n)} (f) = f'(c)$.
\end{prop}

\begin{proof}
Suppose that (i) holds.
For every $n \in \N$, put: 
\[ 
m_1(n) := \frac{\D_{a(n)}^c(f)}{(c_1-a_1(n))(c_2-a_2(n))} \quad 
m_2(n) := \frac{\D_c^{b(n)}(f)}{(b_1(n)-c_1)(b_2(n)-c_2)}
\]
\[ 
m_3(n) := \frac{\D_{(a_1(n),c_2)}^{(c_1,b_2(n))} (f)}{(c_1-a_1(n))(b_2(n)-c_2)} \quad 
m_4(n) := \frac{\D_{(c_1,a_2(n))}^{(b_1(n),c_2)} (f)}{(b_1(n)-c_1)(c_2-a_2(n))}.
\]
From the definition of the double derivative 
we get that $\lim_{n \to \infty} m_i(n) = f'(c)$
for all $i \in \{ 1,2,3,4 \}$.
Thus, from Proposition \ref{propsqueeze}, 
it follows that $\lim_{n \to \infty} m_{a(n)}^{b(n)}(f) = f'(c)$.

If (ii) (or (iii)) holds, then, from the definition of the
double derivative, we get that
$\lim_{n \to \infty} m_{a(n)}^{b(n)}(f) = \lim_{n \to \infty} m_{a(n)}^c(f) = f'(c)$
(or $\lim_{n \to \infty} m_{a(n)}^{b(n)}(f) = \lim_{n \to \infty} m_c^{b(n)}(f) = f'(c)$).
\end{proof}

The next result is \cite[Satz 13]{bogel1934}.
Noteworthily, B\"{o}gel proves this result by an 
argument which is not identical, but similar in 
spirit, to our proof. Thus, in particular, 
for single variable functions,
B\"{o}gel's proof produces the classical Rolle's 
theorem without resorting to the Weierstrass
extremum theorem (cf.~\cite{plante2017}).

\begin{prop}[Double Rolle's theorem]\label{doublerolle}
Let $a,b \in \R^2$ satisfy $a < b$.
Let $f : [a,b] \to \R$ be a 
double continuous function which is double differentiable on 
$(a,b)$. If $\D_a^b(f) = 0$, 
then there exists $c \in (a,b)$ with $f'(c)=0$.
\end{prop}

\begin{proof}
We claim that there are $p,q \in (a,b)$ with 
$\D_p^q(f) = 0$, $p < q$ and $q-p \leq (b-a)/2$.
Let us assume for a moment that the claim holds.
Then we can inductively define sequences
$a(1) \leq a(2) \leq a(3) \leq \cdots$
and $b(1) \geq b(2) \geq b(3) \geq \cdots$
in $(a,b)$ satisfying
$a(n) <  b(n)$, $\D_{a(n)}^{b(n)} = 0$ and
$b(n)-a(n) \leq (b-a)/2^n$, for every $n \in \N$.
Then $\{ a(n) \}_{n=1}^{\infty}$ and $\{ b(n) \}_{n=1}^{\infty}$
have a common limit $c \in (a,b)$ satisfying
 one of the properties (i)-(iii) in Proposition \ref{propsqueezeagain}.
Thus, from the same proposition, it follows that $f'(c)=0$.

Now we show the claim. 
Seeking a contradiction, suppose that
$\D_p^q(f) \neq 0$ for all $p,q \in (a,b)$ such that
$p < q$ and $q-p \leq (b-a)/2$.
Put $h = (b - a)/2$.

\emph{Case 1: $\D_a^{a+h}(f) \neq 0$}.
Consider the map 
$g(x) = \D_x^{x+h}(f)$ for $x \in [a , a+h]$. 
From Proposition \ref{subdivision} it follows that
$g(a) + g(a_1,(b_2+a_2)/2) +  
g( (b_1+a_1)/2, a_2 ) + g( (a+b)/2 ) = 
\D_a^b(f) = 0$.
Therefore, since $g(a) = \D_a^{a+h}(f) \neq 0$,
at least one the real numbers 
$g(a)$, $g(a_1,h_2/2)$, 
$g( (b_1+a_1)/2, a_2 )$ and $g( (a+b)/2 )$
is positive and at least one is negative.
By Proposition \ref{propcontinuity} and Proposition \ref{propintermeiate} it follows that 
there is $p \in (a , (a+b)/2 )$ with $g(p) = 0$. 
Put $q = p + h = p + (b-a)/2$.
Then $\D_p^q(f) = g(p) = 0$, $p,q \in (a,b)$,
$p < q$ and $q-p \leq (b-a)/2$.
This is a contradiction.

\emph{Case 2: $\D_a^{a+h}(f) = 0$}.
Put $b' = a+h$ and $h' = h/2$. 
Then $\D_{a}^{b'}(f) = 0$ and, by the assumptions,
we get that $\D_{a+h'}^{b'}(f) \neq 0$.
Consider the map 
$g(x) = \D_x^{x+h'}(f)$ for $x \in [a , a+h']$.
From Proposition \ref{subdivision} it follows that
$g(a) + g(a_1,(b_2+a_2)/4) +  
g( (b_1+a_1)/4, a_2 ) + g( (a+b)/4 ) =
\D_a^{b'}(f) = 0$.
Therefore, since $g(a) = \D_{a+h'}^{b'}(f) \neq 0$,
at least one the real numbers 
$g(a)$, $g(a_1,(b_2+a_2)/4)$, 
$g( (b_1+a_1)/4, a_2 )$ and $g( (a+b)/4 )$
is positive and at least one is negative.
By Proposition \ref{propcontinuity} and Proposition \ref{propintermeiate} it follows that 
there is $p' \in (a , (a+b)/4 )$ with $g(p') = 0$. 
Put $q' = p' + h' = p' + (b-a)/4$.
Then $\D_{p'}^{q'}(f) = g(p') = 0$, $p',q' \in (a,b)$,
$p' < q'$ and $q'-p' \leq (b-a)/4$.
This is a contradiction.
\end{proof}

\begin{prop}[Double Lagrange's mean value theorem]\label{meanvaluethm}
Let $a,b \in \R^2$ satisfy $a < b$. 
Suppose that $f : [a,b] \to \R$ is a double continuous function 
which is double differentiable on $(a,b)$.
Then there exists $c \in (a,b)$ with $f'(c) = m_a^b(f).$
\end{prop}

\begin{proof}
Consider the map $g(x) = f(x) - m_a^b(f) (x_1-a_1)(x_2-a_2)$
for $x \in [a,b]$. 
Then it is clear that $\D_a^b(g) = \D_a^b(f) - \D_a^b(f) = 0$.
Proposition \ref{doublerolle} implies the existence of an element $c \in (a,b)$ with $g'(c) = 0$.
From Proposition \ref{schwarztheorem} it follows that
$0 = g'(c) = f'(c) - m_a^b(f) \cdot 1$. 
Therefore $f'(c) = m_a^b(f)$.
\end{proof}

\begin{prop}[Double Cauchy's mean value theorem {\cite[Satz 14]{bogel1934}}]\label{meanvaluethmcauchy}
Let $a,b \in \R^2$ satisfy $a < b$. 
Suppose that $f,g : [a,b] \to \R$ are double continuous functions 
which are double differentiable on $(a,b)$.
Then there exists $c \in (a,b)$ with 
$f'(c) \D_a^b(g) = g'(c) \D_a^b(f).$
\end{prop}

\begin{proof}
We consider two cases.

\emph{Case 1: $\D_a^b(g) = 0$.} 
By Proposition \ref{doublerolle} there is $c \in (a,b)$
with $g'(c) = 0$. For this $c$ we get that 
$f'(c) \D_a^b(g) = f'(c) \cdot 0 = 0 \cdot \D_a^b(f) 
= g'(c) \D_a^b(f).$

\emph{Case 2: $\D_a^b(g) \neq 0$.}
Consider the map
$h(x) = f(x) - \D_a^b(f) g(x)/\D_a^b(g)$
for $x \in [a,b]$. Then, clearly, 
$\D_a^b(h) = 0$. By Proposition \ref{doublerolle}
there is $c \in (a,b)$ with $h'(c) = 0$
form which the claim follows.
\end{proof}

\begin{rem}
(a) If we specialize $g(x) = x_1 x_2$ in Proposition 
\ref{meanvaluethmcauchy}, then we get 
Proposition \ref{meanvaluethm}.

(b) In \cite{dobrescu1965} the concept of double derivative
(there named \emph{bidimensional derivative}) as well 
as Propositions \ref{doublerolle}-\ref{meanvaluethmcauchy} 
were rediscovered (independently from Bögel it seems) 
by Dobrescu and Siclovan.
\end{rem}

Suppose that $f$ is a double function defined on a 
double interval $I$.
We say that $f$ is  \emph{double increasing
(double decreasing)} on $I$
if $\D_a^b(f) > 0$ ($\D_a^b(f) < 0$)
for all $a,b \in I$ with $a < b$.

\begin{exa}\label{exaD}
Let $a,b,c \in \R^2$ and $D \in \R$. 
Define a double function $f$ on $\R^2$ by 
$f(x) = D(x_1-c_1)(x_2-c_2)$ for $x \in \R^2$.
Then:
\[
\begin{array}{rcl}
\D_a^b(f) 
&=& D(b_1-c_1)(b_2-c_2) - D(b_1-c_1)(a_2-c_2) \\[5pt]
&-& D(a_1-c_1)(b_2-c_2) + D(a_1-c_1)(a_2-c_2) \\[5pt]
&=& D(b_1-a_1)(b_2-a_2).
\end{array}
\] 
Thus, $f$ is double increasing on $\R^2$ 
$\Leftrightarrow$ $D > 0$;
$f$ is double decreasing on $\R^2$ 
$\Leftrightarrow$ $D < 0$;
$f$ is double constant on $\R^2$ $\Leftrightarrow$ $D = 0$.
\end{exa}

\begin{prop}\label{propdoubleconstant}
Suppose that $a,b \in \R^2$ satisfy $a < b$.
Let $f : [a,b] \to \R$ be a double continuous function 
which is double differentiable on $(a,b)$. Then:
\begin{enumerate}[label={\rm (\alph*)}]

\item $f$ is double increasing on $[a,b]$ 
$\Leftrightarrow$
$f'(x) > 0$ for every $x \in (a,b)$;

\item $f$ is double decreasing on $[a,b]$ 
$\Leftrightarrow$
$f'(x) < 0$ for every $x \in (a,b)$;

\item $f$ is double constant on $[a,b]$ 
$\Leftrightarrow$
$f'(x) = 0$ for every $x \in (a,b)$.

\end{enumerate}
\end{prop}

\begin{proof}
This follows immediately from Proposition \ref{meanvaluethm}. 
\end{proof}

\begin{rem}
Consider the function $f$ defined in Example \ref{exaD}.
Then, by Proposition \ref{schwarztheorem}, it follows that
$f$ is double differentiable with
$f'(x) = D$ for $x \in \R^2$.
Thus, the conclusions in this example follow
from Proposition \ref{propdoubleconstant}.
\end{rem}

Suppose that $f$ is a double function defined on 
a double open interval $I$. Let $a \in I$.
We say that $a$ is a \emph{double maximum point
(double minumum point)} for $f$ on $I$ 
if $\D_a^b(f) < 0$ ($\D_a^b(f) > 0$) for all $b \in I$
with $b \nsim a$.
We say that $a$ is a 
\emph{double extreme point} for $f$ on $I$
if $a$ is a double minumum point or a double 
maximum point for $f$ on $I$

\begin{exa}\label{exampledouble}
Let $a,b,c \in \R^2$ and $D \in \R$. 
Define a double function $f$ on $\R^2$ by 
$f(x) = D(x_1-c_1)^2(x_2-c_2)^2$ for $x \in \R^2$
where $D$ is a non-zero real number.
Then:
\[
\begin{array}{rcl}
\D_a^b(f) 
&=& D(b_1-c_1)^2(b_2-c_2)^2 - D(b_1-c_1)^2(a_2-c_2)^2 \\[5pt]
&-& D(a_1-c_1)^2(b_2-c_2)^2 + D(a_1-c_1)^2(a_2-c_2)^2 \\[5pt]
&=& D(b_1-a_1)(b_2-a_2)(b_1+a_1-2c_2)(b_2+a_2-2c_2).
\end{array}
\] 
Suppose now that $a$ is a double extreme point for $f$
on $\R^2$. 
The above calculation shows that $\D_a^{2c-a} = 0$.
Therefore $a = 2c - a$ and thus $a = c$.
Hence $\D_a^b(f) = D (b_1-a_1)^2 (b_2-a_2)^2$.
Thus, $a$ 
is a double maximum (minimum) point for $f$ on $\R^2$
if and only if $a = c$ and $D < 0$ ($D > 0$).
\end{exa}

Suppose that $f$ is a double function which is defined on a
double open interval $I$ and let $a \in I$.
If $f$ is double differentiable at $a$
and $f'(a)=0$, then $f$ we say that $a$
is a \emph{double stationary point} for $f$.
We say that $a$ is a \emph{double a critical point for $f$} 
if either $f$ is not double differentiable at $a$
or $a$ is a stationary point for $f$.

\begin{prop}[Double Fermat's theorem]\label{propfermat}
Suppose that $f$ is a double function which is
defined on a double open interval $I$ and let $a \in I$.
If $a$ is a double extreme point for $f$ on $I$,
then $a$ is a double critical point for $f$.
\end{prop}

\begin{proof}
Suppose that $f$ is double differentiable at $a$
and that $a$ is a double minimum point for $f$ on $I$.
Then $\D_a^b(f) > 0$ for all $b$ with $b \nsim a$.
Hence:
\[
f'(a) \quad = \lim_{\quad b \leadsto a^{++}} m_a^b(f)
\quad = 
 \lim_{\quad b \leadsto a^{++}} 
\frac{ \Delta_a^b(f) }{(b_1-a_1)(b_2-a_2)} \geq 0
\]
and:
\[
f'(a) \quad = \lim_{\quad b \leadsto a^{+-}} m_a^b(f) 
\quad = 
\lim_{\quad b \leadsto a^{+-}}
\frac{ \Delta_a^b(f) }{(b_1-a_1)(b_2-a_2)} \leq 0.
\]
Thus $f'(a) = 0$.
The proof is analogous in the case when $f$ is double
differentiable at $a$ and $a$ is a double maximum 
point for $f$ on $I$. 
\end{proof}

\begin{prop}[Double first derivative test]\label{propsigns}
Suppose that $f$ is a double \linebreak function which is
double differentiable on a double 
open interval $(a,b)$.
Let $c \in (a,b)$ be a double stationary point for $f$.
\begin{enumerate}[label={\rm (\alph*)}]

\item Suppose that $f'(x) < 0$, when
$a < x < c$ or $c < x < b$, and $f'(x) > 0$ 
when $(a_1,c_2) < x < (c_1,b_2)$ or
$(c_1,a_2) < x < (b_1,c_2)$. Then $c$ is a double maximum
point for $f$ on $(a,b)$.

\item Suppose that $f'(x) > 0$, when
$a < x < c$ or $c < x < b$, and $f'(x) < 0$ when
$(a_1,c_2) < x < (c_1,b_2)$ or
$(c_1,a_2) < x < (b_1,c_2)$. Then $c$ is a double minimum
point for $f$ on $(a,b)$.

\item If $f'$ has the same sign throughout the 
formulation of the statement 
in (a) (or in (b)), then $c$ is neither a double maximum
point nor a double minimum point for $f$ on $(a,b)$.

\end{enumerate}
\end{prop}

\begin{proof}
(a) Take $x \in (a,b)$ with $x \nsim c$.
By Proposition \ref{propdoubleconstant}(b)
it follows that $f$ is double decreasing on the 
intervals $(a,c)$ and $(c,b)$. 
Thus $\D_x^c(f) < 0$ for all $x$ in those intervals.
By Proposition \ref{propdoubleconstant}(a)
it follows that $f$ is double increasing on the 
intervals $((a_1,c_2),(c_1,b_2))$ and 
$((c_1,a_2),(b_1,c_2))$. 
By Lemma \ref{subdivision}(c) it follows that
$\D_x^c(f) > 0$ for all $x$ in
those intervals. 
The proofs of (b) and (c) are similar to the proof of (a).
\end{proof}

\begin{rem}
One can reach the
conclusion in Example \ref{exampledouble},
using the double derivative. 
Namely, let $a$ $b$, $c$, $D$ and $f$ 
be defined as in that example.
By Proposition \ref{schwarztheorem}, it follows that
$f$ is double differentiable with
$f'(x) = 4D(x_1-c_1)(x_2-c_2)$ for $x \in \R^2$.

Suppose that $a$ is a double maximum point for $f$
on $\R^2$. 
By Proposition \ref{propfermat} we get that
$f'(a) = 0$, that is $4D(a_1-c_1)(a_1-c_1) = 0$.
Therefore $a_1 = c_1$ or $a_2 = c_2$.
We consider the case when $a_1 = c_1$
(the case when $a_2 = c_2$ reaches the same conclusion)
so that $f'(x) = 4D(x_1-a_1)(x_2-c_2)$.
By Proposition \ref{propsigns} $f'(x) < 0$ for large 
enough $x$. Thus $D < 0$. Also, by the same proposition,
if $x_1 \neq a_1$, then  $f'(x)$ should change sign 
as $x_2$ goes from a value
less than $a_2$ to a value larger than $a_2$.
Hence $c_2 = a_2$ so that 
$f'(x) = 4D(x_1-a_1)(x_2-a_2)$.
Using this and Proposition \ref{propsigns},
it is easy to see that $a$ now is a double maximum point
for $f$ on $\R^2$. 

A similar analysis reveals that $a$ is double minimum
point for $f$ on $\R^2$ if and only if $a = c$ and $D > 0$.
\end{rem}

\begin{rem}
To the best of our knowledge,
Bögel neither treats double critical points
nor a double Fermat's theorem.
However, in \cite[\S 6]{bogel1935} Bögel studies
\emph{monotone} double functions in the context of 
functions of bounded variation.
\end{rem}

Suppose that $f$ is a double function defined on a
double interval $I$.
We say that a double function $F$ defined on $I$ is a 
\emph{double primitive function of $f$ on $I$} 
if $F$ is double differentiable on $I$ and 
$F'(x) = f(x)$ for $x \in I$. 

\begin{prop}\label{doubleconstant}
Suppose that $f$ is a double function defined on a
double interval $I$.
If $F$ and $G$ are double primitive functions of $f$
on $I$, then there is a double constant function $H$,
defined on $I$, such that $G = F + H$.
\end{prop}

\begin{proof}
Put $H = G - F$. Then $H' = G' - F' = f-f = 0$.
Proposition~\ref{propdoubleconstant}(c) implies that
$H$ is a double constant function.
Clearly $G = F + H$.
\end{proof}

\section{Double integrability}\label{sec:doubleintegrability}

In this section, we define the double Newton integral.
Using the double mean value theorem, we 
obtain a mean value theorem for double 
Newton integrals (see Proposition \ref{meannewton}).
After that, we connect the double Newton \linebreak integral
to the Riemann double integral in 
the first and second double \linebreak fundamental 
theorems of calculus (see Proposition \ref{firstftc}
and Proposition \ref{secondftc}).
At the end of this section, we introduce
improper double Newton integrals.
We also discuss examples of double integrals over
non-rectangular regions.
Most of the material in this section (except the discussion
on improper \linebreak integrals) can be extracted
and specialized from~\cite{bogel1935}. 
However, since we only restrict
ourselves to the double calculus, our presentation
can be more streamlined.

Let $f$ be a double function defined on a
double interval $I$.
Suppose that there exists a double primitive function
$F$ of $f$ on $I$.
Let $a,b \in I$. We
say that the \emph{double Newton integral
of $f$ from $a$ to $b$} is the real number: 
\begin{equation}\label{defNewton}
\int_a^b f := \D_a^b(F).
\end{equation}

\begin{prop}[The double Newton integral is well defined]
The value of \eqref{defNewton} does not depend 
on the choice of the double primitive function.
\end{prop}

\begin{proof}
Let $f,F,G$ be double functions defined on $I$
where $F$ and $G$ are double primitive functions of $f$
on $I$. By Proposition \ref{doubleconstant},
$G = F + H$ for some double constant function
$H$ defined on $I$. Take $a,b \in I$. Then
it follows that
$\D_a^b(G) = \D_a^b(F) + \D_a^b(H) = \D_a^b(F) + 0 =
\D_a^b(F).$
\end{proof}

\begin{exa}
Suppose that $F$ and $f$ are the double functions
defined on $I :=[0,2] \times [1,3]$ by
$F(x) = x_1^2 x_2^3/2$ and 
$f(x) = 3 x_1 x_2^2$ for $x \in I$. 
By Proposition \ref{schwarztheorem},
$F$ is double differentiable on $I$ 
with $F' = f$. Therefore:
\[
\displaystyle{\int_{(0,1)}^{(2,3)} f } \ = \ 
\D_{(0,1)}^{(2,3)}(F) 
\ = \ F(2,3) - F(2,1) - F(0,3) + F(0,1) \ = \ 52. 
\]
\end{exa}

\begin{prop}[Properties of the double Newton integral]\label{propertiesintegral}
Let $f$ be a double function defined on a 
double interval~$I$.
Suppose that there exists a double primitive function $F$
of $f$ on $I$. If $a,b,c \in I$, then:
\begin{enumerate}[label={\rm (\alph*)}]

\item $\displaystyle{ \int_a^a f = 0 }$; \ 
$\displaystyle{ \int_a^b f = \int_b^a f }$; \ 
$\displaystyle{ \int_a^b f =-\int_{(a_1,b_2)}^{(b_1,a_2)} f }$;

\item $\displaystyle{ \int_a^b f = \int_a^{(c_1,b_2)} f + 
\int_{(c_1,a_2)}^b f }$;

\item $\displaystyle{ \int_a^b f = \int_a^{(b_1,c_2)} f + 
\int_{(a_1,c_2)}^b f }$;

\item $\displaystyle{ \int_a^b f = 
\int_a^c f + 
\int_c^b f + 
\int_{(a_1,c_2)}^{(c_1,b_2)} f +  
\int_{(c_1,a_2)}^{(b_1,c_2)} f }$.

\end{enumerate}
\end{prop}

\begin{proof}
This follows immediately from Proposition \ref{subdivision}. 
\end{proof}

\begin{prop}[The double Newton integral is a primitive function]\label{Dnewtonfunction} 
Let $f$ be a double function defined on a 
double interval~$I$.
Suppose that there exists a double primitive function $F$
of $f$ on $I$. Take $a \in I$ and define the 
map $G : I \to \R$ by $G(x) \mapsto \int_a^x f$ for $x \in I$.
\begin{enumerate}[label={\rm (\alph*)}]

\item The identity $\displaystyle{ \D_b^x(G) = \D_b^x(F) }$
holds for all $b,x \in I$.

\item The function $G$ is double differentiable on $I$ and
$\displaystyle{ G'(b) = f(b) }$ for $b \in I$.

\end{enumerate}
\end{prop}

\begin{proof}
Take $a,b,x \in I$.
First we show (a). By Proposition 
\ref{propertiesintegral} we get that:
\[
\begin{array}{rcl}

\D_b^x(G) & = & \displaystyle{ \int_a^x f + \int_a^b f -
\int_a^{(x_1,b_2)} f - \int_a^{(b_1,x_2)} f } \\[10pt]
& = & \displaystyle{ \int_a^x f + \int_b^a f +
\int_{(a_1,b_2)}^{(x_1,a_2)} f + 
\int_{(a_1,x_2)}^{(b_1,a_2)} f } 
\ = \ \displaystyle{ \int_b^x f \ = \ \Delta_b^x(F) }.
\end{array}
\]
Next we show (b). By (a) we get that:
\[
\begin{array}{rcl}
G'(b) & = & \displaystyle{ \lim_{x \leadsto b} m_b^x(G) } 
\ = \ \displaystyle{ \lim_{x \leadsto b} 
\frac{ \D_b^x(G) }{ (x_1-b_1)(x_2-b_2) } } \\[10pt] 
& = & \displaystyle{ \lim_{x \leadsto b} 
\frac{ \D_b^x(F) }{ (x_1-b_1)(x_2-b_2) } } 
\ = \ \displaystyle{ \lim_{x \leadsto b} 
m_b^x(F) \ = \ F'(b) \ = \ f(b). }
\end{array}
\]
Alternatively, the equality $G'=f$ follows from the 
fact that $G$ equals $F$ plus the double constant function
$I \ni (x_1,x_2) \mapsto F(a_1,a_2)-F(x_1,a_2)-F(a_1,x_2)$
which by Proposition \ref{propsplit} has zero double
derivative.  
\end{proof}

\begin{prop}[The mean value theorem for double Newton integrals]\label{meannewton}
Let $f$ be a double function defined 
on a double interval $[a,b]$ for some 
$a,b \in \R^2$ with $a < b$.
Suppose that
there exists a double primitive function $F$ of $f$
on $[a,b]$.
Then there exists $c \in (a,b)$ such that:
\[
\int_a^b f = f(c) (b_1-a_1)(b_2-a_2).
\]
\end{prop}

\begin{proof}
This follows from Propositions \ref{ddiffimpliesdcont}, 
\ref{meanvaluethm} and \ref{Dnewtonfunction}.
\end{proof}

We now recall some classical notions (cf. e.g. \cite{courant}).
Suppose that $f$ is a double function defined 
on a double interval $[a,b]$ for some 
$a,b \in \R^2$ with $a < b$.
A \emph{partition} $P$ of $[a,b]$
is a choice of points
$x_{1,0},x_{1,1},\ldots,x_{1,m} \in [a_1,b_1]$ and 
$x_{2,0},x_{2,1},\ldots,x_{2,n} \in [a_2,b_2]$ such that
$
a_1 = x_{1,0} < x_{1,1} < \cdots < x_{1,m-1} < 
x_{1,m} = b_1
$
and
$
a_2 = x_{2,0} < x_{2,1} < \cdots < x_{2,n-1} < 
x_{2,n} = b_2.
$
Given $P$, define the $mn$ rectangles 
$R_{ij} = [x_{1,i-1},x_{1,i}] \times [x_{2,j-1},x_{2,j}]$, 
for $1 \leq i \leq m$ and $1 \leq j \leq n$.
The \emph{norm} $|P|$ of $P$ is the largest of the diagonals
in these $mn$ rectangles.
Pick an arbitrary point 
$(x_{1,i,j}^* , x_{2,i,j}^*)$ in each 
of the rectangles $R_{ij}$.
For all $i$ and $j$ put
$\D x_{1,i} = x_{1,i} - x_{1,i-1}$ and
$\D x_{2,j} = x_{2,j} - x_{2,j-1}$.
The corresponding \emph{double Riemann sum} is defined as
$
R(f,P) := \sum_{i=1}^m \sum_{j=1}^n 
f( x_{1,i,j}^* , x_{2,i,j}^* ) \D x_{1,i} \D x_{2,j}.
$
The double function $f$ is said to be 
\emph{Riemann integrable} over $[a,b]$ and have 
\emph{double integral}
$
I = \iint_{[a,b]} f(x) \ dx_1 dx_2,
$
if for every $\epsilon \in \R_+$ there is 
$\delta \in \R_+$ such that 
$| R(f,P) - I | < \epsilon$ holds for every 
partition $P$ of $[a,b]$ satisfying $|P| < \delta$
and for all choices of $( x_{1,i,j}^* , x_{2,i,j}^* )$
in the subrectangles of $P$.
If $f$ is continuous on $[a,b]$, then $f$ is 
Riemann integrable over $[a,b]$ 
(see e.g. \cite[p. 293]{courant}).

\begin{prop}[The mean value theorem for double Riemann integrals]\label{propcourant}
Let $a,b \in \R^2$ and $a < b$.
Suppose that $f : [a,b] \to \R$ is a double function
which is continuous.
Then there exists $c \in (a,b)$ such that: 
\[
\iint_{[a,b]} f(x) \ dx_1 dx_2 = f(c)(b_1-a_1)(b_2-a_2).
\]
\end{prop}

\begin{proof}
See e.g. \cite[p. 292]{courant}.
\end{proof}

\begin{prop}[The first double fundamental theorem of calculus]\label{firstftc}
Let $a,b \in \R^2$ and $a < b$.
Suppose that $f : [a,b] \to \R$ is a double function
which is continuous.
Define $G : [a,b] \to \mathbb{R}$ by 
$G(x) = \iint_{[a,x]} f(x) \ dx_1 dx_2$ for $x \in [a,b]$.
Then $G$ is double differentiable on $[a,b]$ with $G' = f$. 
\end{prop}

\begin{proof}
Take $x \in [a,b]$. 
We consider four cases.

Case 1: $x < b$. 
Take $h \in \R_{++}^2$ such that $x+h < b$.
By Proposition \ref{propcourant}:
\[
m_x^{x+h}(G) = 
\frac{ \iint_{[x,x+h]} f(x) \ dx_1 dx_2 }{h_1 h_2} = f(t)
\]
for some $t \in (x,x+h)$.
Letting $h \to 0^{++}$ yields that $G'_{++}(x) = f(x)$.

Case 2: $a < x$.
Take $h \in \R_{--}^2$ such that $a < x+h$.
By Proposition \ref{propcourant}:
\[
m_x^{x+h}(G) = 
\frac{ \iint_{[x+h,x]} f(x) \ dx_1 dx_2 }{(-h_1)(-h_2)} = f(t)
\]
for some $t \in (x+h,x)$.
Letting $h \to 0^{++}$ yields that $G'_{--}(x) = f(x)$.

Case 3: $(a_1,x_2) < (x_1,b_2)$.
Take $h  \in \R_{-+}^2$ such that $x_1+h_1 > a_1$
and $x_2 + h_2 < b_2$. 
By Proposition \ref{propcourant}:
\[
m_x^{x+h}(G) = 
\frac{ \iint_{ [x_1+h_1,x_1] \times [x_2,x_2+h_2] } 
f(x) \ dx_1 dx_2 }{(-h_1)h_2} = f(t)
\]
for some $t \in (x_1+h_1,x_1) \times (x_2,x_2+h_2)$.
Letting $h \to 0^{-+}$ yields that $G'_{-+}(x) = f(x)$.

Case 4: $(x_1,a_2) < (b_1,x_2)$.
Take $h  \in \R_{+-}^2$ such that $x_1+h_1 < b_1$
and $x_2 + h_2 > a_2$. 
By Proposition \ref{propcourant}:
\[
m_x^{x+h}(G) = 
\frac{ \iint_{ [x_1,x_1+h_1] \times [x_2+h_2,x_2] } 
f(x) \ dx_1 dx_2 }{h_1(-h_2)} = f(t)
\]
for some $t \in (x_1,x_1+h_1) \times (x_2+h_2,x_2)$.
Letting $h \to 0^{+-}$ yields that $G'_{+-}(x) = f(x)$.

Cases 1-4 show that $G$ is double differentiable 
on $[a,b]$ with $G' = f$. 
\end{proof}

\begin{prop}[The second double fundamental theorem of calculus]\label{secondftc}
Let $f$ be a double function defined 
on a double interval $[a,b]$ for some 
$a,b \in \R^2$ with $a < b$.
Suppose that
there exists a double primitive function $F$ of $f$
on $[a,b]$. If $f$ is Riemann integrable on
$[a,b]$, then 
$\iint_{[a,b]} f(x) \ dx_1 dx_2 = \int_a^b f$.
\end{prop}

\begin{proof}
Take $\epsilon \in \R_+$ and put $I := \int_a^b f$. 
Since $[a,b]$ is a compact interval, $f$ is 
uniformly continuous on $[a,b]$
(see e.g. \cite[Theorem 4.19]{rudin1976}).
Therefore, there exists $\delta \in \R_+$ such that:
\[
|f(c)-f(d)| < \frac{\epsilon}{ (b_1-a_1)(b_2-a_2) }
\]
whenever $c,d \in [a,b]$ and 
$\sqrt{( c_1-d_1 )^2 + ( c_2-d_2 )^2} < \delta$.
Consider a fixed double Riemann sum:
\[
R := \sum_{i=1}^m \sum_{j=1}^n 
f( x_{1,i,j}^* , x_{2,i,j}^* ) \D x_{1,i} \D x_{2,j}
\]
defined by a partition $P$, with $|P| < \delta$, 
and a choice of points 
$(x_{1,i,j}^* , x_{2,i,j}^*)$ in the corresponding 
rectangles. We wish to show that $| R - I | < \epsilon$.
By Proposition \ref{meannewton} there exist  
$c_{i,j} \in (x_{1,i-1},x_{1,i}) \times (x_{2,j-1},x_{2,j})$ 
with:
\[ \int_{( x_{1,i-1},x_{2,j-1} )}^{ ( x_{1,i},x_{2,j} ) } f = 
f(c_{i,j}) \D x_{1,i} \D x_{2,i} 
\]
for $i=1,\dots,m$ and $j=1,\ldots,n$.
By repeated application of Proposition 
\ref{propertiesintegral}(b)(c) we therefore get that:
\[
I = 
\sum_{i=1}^m \sum_{j=1}^n 
\int_{( x_{1,i-1},x_{2,j-1} )}^{ ( x_{1,i},x_{2,j} ) } f =
\sum_{i=1}^m \sum_{j=1}^n 
f(c_{i,j}) \D x_{1,i} \D x_{2,i} 
\]
which in turn implies that:
\[
\begin{array}{rcl}
\displaystyle{ \left| R - I \right| } &=& 
\displaystyle{
\left| \sum_{i=1}^m \sum_{j=1}^m 
( f(x_{1,i,j}^* , x_{2,i,j}^*) - f(c_{i,j}) )  
\D x_{1,i} \D x_{2,i}  \right| } \\[20pt]
& \leq & \displaystyle{ \sum_{i=1}^m \sum_{j=1}^m 
\left| f(x_{1,i,j}^* , x_{2,i,j}^*) - f(c_{i,j}) \right|  
\D x_{1,i} \D x_{2,i} } \\[15pt]
& < & \displaystyle{ \frac{\epsilon}{ (b_1-a_1)(b_2-a_2) }
\sum_{i=1}^m \sum_{j=1}^m \D x_{1,i} \D x_{2,i} }  \\[15pt]
& = & \displaystyle{ 
 \frac{\epsilon}{ (b_1-a_1)(b_2-a_2) } \cdot
 (b_1-a_1)(b_2-a_2)  = \epsilon. } \\[10pt]
\end{array}
\]
We have now shown that 
$\iint_{[a,b]} f(x) \ dx_1 dx_2 = \int_a^b f$.
\end{proof}

Let $f$ be a double function defined 
on a double \emph{open} interval $(a,b)$ for some 
$a,b \in \overline{\R}^2$ with $a < b$.
Suppose that
there exists a double primitive function $F$ of $f$ on $(a,b)$. 
If the double signed limit:
\begin{equation}\label{doublelimit}
\lim_{ y \to b^{--}; \ x \to a^{++} } \D_x^y(F)
\end{equation}
exists, then we say that $\int_a^b f$ is a 
\emph{convergent} improper double Newton \linebreak 
integral with 
\emph{value} equal to the limit (\ref{doublelimit}).
If the limit in (\ref{doublelimit}) does not exist, 
then we say that 
the improper double Newton integral $\int_a^b f$ 
is \emph{divergent}.
Note that if all of the following four signed limits exist:
\[
\begin{array}{lcl}
\displaystyle{ A := \lim_{ x \to b^{--}} F(x) } & &
\displaystyle{ B := \lim_{ x \to a^{++}} F(x) } \\[15pt]
\displaystyle{ C := \lim_{ x \to (b_1,a_2)^{-+}} F(x) } & &
\displaystyle{ D := \lim_{ x \to (a_1,b_2)^{+-}} F(x) }
\end{array}
\]
then $\int_a^b f$ is convergent with value equal to
$A+B-C-D$.

\begin{exa}
(a) Suppose that $F$ and $f$ are the double functions
defined on $I :=(0,1) \times (0,1)$ by
$F(x) = (x_1+x_2) \ln(x_1+x_2)$ and 
$f(x) = 1/(x_1+x_2)$ for $x \in I$. 
By Proposition \ref{schwarztheorem},
$F$ is double differentiable on $I$ with $F' = f$. 
Since
$\lim_{ x \to (1,1)^{--}} F(x_1,x_2) = 2 \ln(2)$ and
$\lim_{ x \to (0,0)^{++}} F(x_1,x_2) = 
\lim_{ x \to (1,0)^{-+}} F(x_1,x_2) = 
\lim_{ x \to (0,1)^{+-}} F(x_1,x_2) = 0$
the improper double Newton integral
$\int_{(0,0)}^{(1,1)} 1/(x_1+x_2)$
is convergent with value equal to
$2 \ln(2)$.

(b) Suppose that $F$ and $f$ are the double functions
defined on $I :=(0,1] \times (0,1]$ by
$F(x) = x_1/(x_1+x_2)$ and 
$f(x) = (x_1-x_2)/(x_1+x_2)^3$ for $x \in I$. 
By Proposition \ref{schwarztheorem},
$F$ is double differentiable on $I$ with $F' = f$. 
However, since
$\lim_{s \to 0^+} \D_{(s,s)}^{(1,1)}(F) = 0$ and
$\lim_{t \to 0^+} \D_{(t,2t)}^{(1,1)}(F) = -1/6$
the improper double Newton integral
$\int_{(0,0)}^{(1,1)} (x_1-x_2)/(x_1+x_2)^3 $
is divergent.
\end{exa}

Many standard calculus textbook problems concerning double 
integrals over non-rectangular regions 
are solved by iterated integration.
If, however, the region in question can 
be mapped bijectively onto a double interval,
then such integrals can instead be considered as 
improper double Newton integrals. In fact,
by Proposition \ref{secondftc} and the result 
in \cite{schwartz1954} we
get the following:

\begin{prop}
Let $a,b \in \overline{\R}^2$ with $a < b$.
Suppose that $D$ is an open subset of $\R^2$ and that
$h : (a,b) \to D$ is a bijection such that $h$ and $h^{-1}$
are continuous and have continuous partial derivatives.
Let $J(x)$ denote the absolute value 
of the Jacobian determinant of $h$ at $x \in (a,b)$.
If $f : D \to \R$ is a function which is 
integrable on $D$ and $g := (f \circ h) \cdot J$ has 
a double primitive function on $(a,b)$, then 
the improper double Newton integral of $g$ from 
$a$ to $b$ is convergent and
$\iint_{D} f(x) \ dx_1 dx_2 = \int_a^b g$.
\end{prop}

\begin{exa}
(a) We wish to evaluate $\iint_D f(x) \ dx_1 dx_2$
where $f(x) = x_1 x_2$ for $x \in D$, and $D$ is 
the interior of the triangle with vertices 
$(0,0)$, $(1,0)$ and $(1,1)$.
Define $h : (0,1) \times (0,1) \to D$ by
$h(u,v) = (u,uv)$. Then $J(u,v) = u$ and 
$\iint_D f(x) \ dx_1 dx_2 = 
\int_{(0,0)}^{(1,1)} u \cdot uv \cdot u =
\int_{(0,0)}^{(1,1)} u^3 v = 
\D_{(0,0)}^{(1,1)} ( u^4 v^2/8 ) = 1/8$.

(b) We wish to evaluate $\iint f(x) \ dx_1 dx_2$
where $f(x) = 1/(x_1+x_2)^2$ for $x \in D$, and 
$D$ is defined by $0 < x_1 < 1$ and $0 < x_2 < x_1^2$.
Define $h : (0,1) \times (0,1) \to D$ by
$h(u,v) = (u,u^2 v)$ for $u,v \in (0,1)$.
Then $J(u,v) = u^2$ and 
$\iint_D f(x) \ dx_1 dx_2 = 
\int_{(0,0)}^{(1,1)} u^2/(u + vu^2)^2 =
\int_{(0,0)}^{(1,1)} 1/(1+uv)^2 = $ \linebreak
$\D_{(0,0)}^{(1,1)} ( \ln(1+uv) ) = \ln(2)$.

(c) We wish to evaluate $\iint f(x) \ dx_1 dx_2$
where $f(x) = e^{-x_1^2}$ for $x \in D$, and 
$D$ is defined by $x_1 > 0$ and $-x_1 < x_2 < x_1$. 
Define $h : (0,\infty) \times (-1,1) \to D$ by
$h(u,v) = (u,uv)$ for $u \in (0,\infty)$ and $v \in (-1,1)$.
Then $J(u,v) = u$ and 
$\iint_D f(x) \ dx_1 dx_2 = 
\int_{(0,-1)}^{(\infty,1)} e^{-u^2} u = 
\lim_{s \to \infty} 
\D_{(0,-1)}^{(s,1)} ( -e^{-u^2}v/2 ) = 1$.

(d) We wish to evaluate $\iint f(x) \ dx_1 dx_2$
where $f(x) = 1/(x_1+x_2)$ for $x \in D$, and 
$D$ is defined by $1 < x_1$ and 
$0 < x_2 < 1/x_1$. 
Define $h : (1,\infty) \times (0,1) \to D$ by
$h(u,v) = (u,v/u)$ for $u \in (1,\infty)$ and 
$v \in (0,1)$. Then $J(u,v) = u^{-1}$ and 
$\iint_D f(x) \ dx_1 dx_2 = 
\int_{(1,0)}^{(\infty,1)} (u + v/u)^{-1} \cdot u^{-1} =
\int_{(1,0)}^{(\infty,1)} (u^2+v)^{-1} =
\lim_{t \to \infty; \ s \to 0^+} 
\D_{(1,s)}^{(t,1)} \left( u \ln(u^2+v) + 2 \sqrt{v}
\arctan(u/\sqrt{v}) \right) = \pi/2 - \ln(2)$.
\end{exa}

\section{Triple calculus, quadruple calculus and beyond}\label{sec:triplecalculus}

It is easy to work out the corresponding difference 
operators in higher dimensions. 
Namely, if $a,b \in \R^n$ and $f$ is an 
$n$-variable function, then:
\[
\D_a^b(f) \ \ = \ \sum_{s \in \{ 0,1 \}^n } (-1)^s
f( s_1 a_1 + (1-s_1) b_1, \ldots , s_n a_n + (1-s_n)b_n )
\]
where $(-1)^s := (-1)^{s_1 + s_2 + \cdots + s_n}$.
So, for instance, if $n=3$, then we get that:
\[
\begin{array}{rcl}
\D_a^b(f) &=& f(b_1,b_2,b_3) 
\ - \ f(a_1,b_2,b_3) \ - \ f(b_1,a_2,b_3) \ 
- \ f(b_1,b_2,a_3) \\[5pt]
&+& f(a_1,a_2,b_3) \ + \ f(a_1,b_2,a_3) \ + \ f(b_1,a_2,a_3)
\ - \ f(a_1,a_2,a_3).
\end{array}
\]
In \cite{bogel1934,bogel1935}
higher-dimensional analogues of all the results 
established in this article are shown to hold, that is
there is also a triple calculus, a 
quadruple calculus and beyond, at our disposal.

It seems to the author of the present article
that Bögel did not consider higher-dimensional
versions of Schwarz's theorem.
To be more precise, suppose,
for instance, that $f$ is a five-variable function,
that $f$ is double differentiable with respect
to the first two  variables, with double 
derivative denoted by $f'_{12}$, and that 
$f$ is triple differentiable with respect to the 
last three variables, with triple derivative
denoted by $f'_{345}$. If we also suppose that the iterated
derivatives $(f'_{345})'_{12}$ and 
$(f'_{12})'_{345}$ exist and are double 
respectively triple continuous, does
it then follow that $f$ is quintuple differentiable
with $f'_{12345} = (f'_{345})'_{12} = (f'_{12})'_{345}$?
Since the proof of Proposition \ref{schwarztheorem}
only depends on the mean value theorem in each variable,
it seems reasonable to believe that this proof
is generalizable to higher dimensions if we use
Bögel's higher-dimensional mean value theorem. 

Another  classical result from calculus that neither we
nor B\"{o}gel has considered
is Darboux's theorem. Recall that this result 
states that if a single variable function is differentiable on 
an open interval, then the derivative enjoys the 
intermediate value property on this interval.
It is not clear to the author of the present 
article if there is a double (or higher-dimensional) 
analogue of this result.
Note that the usual text book proof for Darboux's theorem uses 
the Weierstrass extreme value theorem 
(see e.g. \cite{olsen2004}),
which, as we have pointed out earlier, is not at
our disposal for double functions.
However, there are proofs of Darboux’s theorem 
which are based only on the mean value theorem
for differentiable functions and the intermediate 
value theorem for continuous functions
(see loc. cit.). Therefore, it is plausible that
there indeed is  
a double (and higher) version(s) of Darboux's theorem
which is (are) reachable by the methods used in this article.

\end{document}